\documentclass[11pt]{amsart}
\usepackage[left=2cm,top=2cm,right=3cm,nofoot]{geometry}
\usepackage{amsmath,mathtools}%
\usepackage{amsfonts}%
\usepackage{amssymb}%
\usepackage{graphicx}
\usepackage{esint}
\usepackage{hyperref}
\usepackage{enumitem}
\usepackage{accents}
\usepackage{etoolbox}
\patchcmd{\subsection}{-.5em}{.5em}{}{}
\newtheorem{theorem}{Theorem}[section]
\theoremstyle{plain}

\newtheorem{lemma}[theorem]{Lemma}

\newtheorem{remark}[theorem]{Remark}

\numberwithin{equation}{section}
\theoremstyle{plain}
\newtheorem{claim}{Claim}
\usepackage{etoolbox}
\AtEndEnvironment{proof}{\setcounter{claim}{0}}

\begin{document}

\title[]{Spinor inequality for magnetic fields on spin manifolds}

\author{Jurgen Julio-Batalla}
\address{ Universidad Industrial de Santander, Carrera 27 calle 9, 680002, Bucaramanga, Santander, Colombia}
\email{ jajuliob@uis.edu.co}
\thanks{ }

\begin{abstract} 
This paper is concerned with the zero mode equation $D_g\varphi=iA\cdot\varphi$ on closed spin manifold $(M^n,g,\sigma)$ of positive scalar curvature. Here $A$ is a real one form on $M$. We proved that if $(\varphi, A)$ is a non trivial solution of the zero mode equation then $$\parallel dA\parallel_{n/2}>Y(M^n,[g])/(4v_n^{1/2}),$$ where $Y(M^n,[g])$ is the Yamabe constant of $(M^n,g)$ and $v_n=\left[\frac{n}{2}\right]$. In the case of the round sphere $(\mathbb{S}^n,g_{can},\sigma_{can})$ this result confirms that the inequality obtained in \cite{Frank} is not sharp.
\end{abstract}

\maketitle

\section{Introduction}
On closed spin manifold $(M^n,g,\sigma)$ we are interested in the zero mode equation
\begin{equation}\label{Zero}
D_g\varphi=iA\cdot_g\varphi,
\end{equation}

where $A$ is a real one form, $\varphi\in \Gamma(SM)$ and $D_g$ is the Dirac operator with respect to the metric $g$. The action ``$\cdot_g$" of the one form $A$ is via Clifford multiplication.

The study of this equation comes from several areas of mathematical physics. It is mostly associated with the zero mode equation in $\mathbb{R}^3$ with respect to the Euclidean metric $dx^2$. For instance, such equations appear in the study of atoms interacting with magnetic fields, where their presence can signal instability. Specifically, in the nonrelativistic description of a hydrogen-like atom in a magnetic field, the Hamiltonian includes a spin–magnetic field interaction term. By adding the magnetic field energy to the system, Fröhlich-Lieb-Loss in \cite{Frohlich} obtained a variational problem whose infimum can become arbitrarily negative (indicating instability) if the nuclear charge $Z$ exceeds a critical value. This instability is intimately related to the existence of a nontrivial solution to the zero mode equation with finite energy magnetic field. For more details on this physical problem, we refer the reader to \cite{Frohlich}.

The first explicit example of a zero mode was constructed  by Loss-Yau in \cite{Loss-Yau}. Since then, considerable effort has been devoted to understanding the structure of the equation \eqref{Zero}. For an overview of recent developments on zero mode equation, we refer the reader to   \cite{Frank} and the references therein.

In \cite{Frank} Frank-Loss addressed the question of absence of zero modes in a quantitative way, establishing a necessary condition on the size of a magnetic field $dA$ for the existence of a non-trivial zero mode $(\varphi,A)$. More precisely, if $(\varphi,A)$ is a zero mode in $\mathbb{R}^3$ with respect to Euclidean metric $dx^2$, they proved that 
\begin{equation}\label{notsharp}
\left(\int_{\mathbb{R}^3}|dA|^{\frac{3}{2}}dx\right)^{\frac{2}{3}}\geq 2S_3 ,
\end{equation}

where $S_3$ is the Sobolev constant in $\mathbb{R}^3$.

In the absence of explicit examples attaining equality, Frank-Loss in \cite{Frank}  raised the question of whether estimate \eqref{notsharp} is sharp. They conjectured that the sharp constant might be $4S_3$.

The main goal of this paper is to show that inequality \eqref{notsharp} is not optimal. Our result is the following

\begin{theorem}
Let $(M^n,g,\sigma)$ be a closed spin manifold of positive scalar curvature $s_g$. If $(\varphi,A)$ is a nontrivial zero mode then $$\parallel dA\parallel_{n/2}>Y(M^n,[g])/(4v_n^{1/2}),$$ where $Y(M^n,[g])$ is the Yamabe constant of $(M^n,g)$ and, $v_n=(n-1)/2$ if $n$ is odd and $v_n=n/2$ if $n$ is even.
\end{theorem}

\begin{remark}
Since the zero mode equation \eqref{Zero} is conformally invariant the estimate \eqref{notsharp} also holds for zero modes on the round sphere $(\mathbb{S}^3,g_{can},\sigma_{can})$. Notice that $(4v_3^{1/2})^{-1}Y(\mathbb{S}^3,[g_{can}])=\frac{3}{2}vol(\mathbb{S}^3,g_{can})^{2/3}=2S_3$.
\end{remark}

\begin{remark}
In \cite{Dunne} Dunne-Min generalized to higher odd dimensions the zero mode $(\varphi,A)$ constructed by Loss-Yau. Frank-Loss \cite{Frank-Loss2} later proved that these zero modes are optimizers for the sharp inequality $$\parallel A\parallel_n^2\geq \frac{n}{(n-2)}S_n.$$
It would be interesting to investigate whether they also optimize a sharp inequality for $\parallel dA\parallel_{n/2}$.
\end{remark}

Regarding regularity, due to regularity results in \cite{Frank,Frank-Loss2,Wang} for zero modes $(\varphi,A)$, we assume throughout that $|A|\in L^n(M^n)$ and the spinor field $\varphi$ belongs to $L^p(M^n)$ for some $p>n/(n-1)$. 

The key ingredient in \cite{Frank} for the proof of \eqref{notsharp} is an improved version of the diamagnetic inequality $|\nabla|\varphi||^2\leq \frac{2}{3}|(-i\nabla-A)\varphi|^2$ for zero modes. Our proof follows a different approach. Our key ingredient is a new conformal weighted invariant
$$\lambda_1\left(4i\frac{\langle dA\cdot_{g}\varphi,\varphi\rangle_{g}}{|\varphi|_{g}^2},g\right),$$
defined as the first eigenvalue of the weighted linear  problem 
$$L_g(u)=\mu \;4i\frac{\langle dA\cdot_g\varphi,\varphi\rangle_g}{|\varphi|_g^2}u,$$
where $L_g$ is the conformal Laplacian for the metric $g$. Further details on this invariant are given in Section 3.

Recently, conformal weighted invariants with different choices of weight have been used to characterize equality in some sharp spinorial inequalities. For instance, the author in \cite{Julio} defined the conformal weighted invariant with respect to the weight $|\varphi|_g^{4/(n-1)}$. When $\varphi$ is a solution of the spinorial Yamabe equation
$$D_g\varphi=|\varphi|^{2/(n-1)}\varphi,$$
the author characterized equality in the sharp inequality $$\lambda^+_{min}(M^n,[g],\sigma)^2\geq \frac{n}{4(n-1)}Y(M^n,[g]).$$ 
Here $\lambda_{min}^+$ is the Bär-Hijazi-Lott invariant.

Very recently, Wang-Zhang in \cite{Wang} considered the conformal weighted invariant with respect to the weight $|A|_g^2$. When $(\varphi,A)$ is a zero mode, they characterized equality in $$\parallel A\parallel_n^2\geq \frac{n}{4(n-1)}Y(M^n,[g]).$$

The plan of this paper is as follows. In Section 2, we recall some background on spin manifolds and the magnetic Dirac operator. We also discuss the invariance properties of equation (1.1). In Section 3, we introduce the conformal weighted invariant $\lambda_1\left(4i\frac{\langle dA\cdot_{g}\varphi,\varphi\rangle_{g}}{|\varphi|_{g}^2},g\right)$ and provide the proof of Theorem 1.1.

\section{Preliminaries}
\subsection{Spin manifolds}
We recall some basic facts about spin manifolds. More details can be found in \cite{HijaziBook}

On a closed oriented Riemannian manifold $(M^n,g)$ we can define a $SO(n)$-principal bundle $P_{SO}M$ over $M$ of oriented $g-$orthonormal bases at $x\in M$. For $n\geq 3$, there exists the universal covering $\sigma:spin(n)\rightarrow SO(n)$ where $spin(n)$ is the group generated by even unit-length vector of $\mathbb{R}^n$ in the real Clifford algebra $Cl_n$ (the associative $\mathbb{R}-$algebra generated by relation $VW+WV=-2( V,W)$ for the Euclidean metric $(,)$). The manifold $M$ is called spin if there is a $spin(n)-$principal bundle $P_{spin}M$ over $M$ such that it is a double covering of $P_{SO}M$ whose restriction to each fiber is 
the double covering $\sigma:spin(n)\rightarrow SO(n)$. Such a double covering from $P_{spin}M$ to $P_{SO}M$, $\sigma$, is 
known as a spin structure. 

\medskip

There are  four special structures associated to a spin manifold $(M^n,g,\sigma)$:
\begin{enumerate}
\item A complex vector bundle $SM:=P_{spin}(M)\times_{\rho}\Sigma_n$ where $\rho:spin(n)\rightarrow Aut(\Sigma_n)$ is the restriction to $spin(n)$ of an irreducible representation $\rho:\mathbb{C}l_n\rightarrow End(\Sigma_n) $ of the complex Clifford algebra $\mathbb{C}l_n\simeq Cl_n\otimes_{\mathbb{R}} \mathbb{C}$,  $\Sigma_n\simeq\mathbb{C}^N$ and $N=2^{[n/2]}$.
\item  Clifford multiplication $m$ on $SM$ defined by
\begin{align*}
m:TM\times SM&\rightarrow SM\\
X\otimes\varphi&\mapsto X\cdot_g\varphi:=\rho(X)\varphi.
\end{align*} 
\item A Hermitian product $\langle\cdot,\cdot\rangle$ on sections of $SM$.

\item The Levi-Civita connection $\nabla$ on $SM$.
\end{enumerate}

These structures are compatible in the following sense:
\begin{align*}
\langle X\cdot\varphi,\psi\rangle&=-\langle\varphi,X\cdot\psi\rangle, \\
X(\langle\varphi,\psi\rangle)&=\langle\nabla_X\varphi,\psi\rangle+\langle\varphi,\nabla_X\psi\rangle,\\
\nabla_X(Y\cdot\varphi)&=\nabla_XY\cdot\varphi+Y\cdot\nabla_X\varphi,
\end{align*}
for all $X,Y\in\Gamma(TM)$ and $\varphi,\psi\in\Gamma(SM)$.
Given the Levi-Civita  connection $\nabla:\Gamma(SM)\rightarrow\Gamma(Hom(TM,SM))$ and identifying $\Gamma(Hom(TM,SM))$ with $\Gamma(TM\otimes SM)$, we can define the Dirac operator $D_g$ as the composition of $\nabla$ with the Clifford multiplication $m$ i.e. $D_g:=m\circ \nabla$. For a local orthonormal frame $\{E_j\}$ we have $$D_g\varphi=\sum\limits_{j=1}^nE_j\cdot_g\nabla_{E_j}\varphi.$$

For later purpose, we briefly recall some properties and identities of this operator.
 
For the square of Dirac operator we have the Schr\"odinger-Lichnerowicz formula
$$D_g^2=\nabla^*\nabla+\frac{1}{4}sc_gId_{\Gamma(SM)},$$
where $\nabla^*:\Gamma(Hom(TM,SM))\rightarrow\Gamma(SM)$ is the adjoint of the Levi-Civita connection $\nabla$.

There are also Schr\"odinger-Lichnerowicz formulas for different variation of the operator $D_g$. Let $A$ be a smooth real one form on $M^n$. The magnetic Dirac operator with magnetic potential $A$ is defined by $$D_g-iA\cdot_g:\Gamma(SM)\rightarrow\Gamma(SM)$$
and the spinorial magnetic connection Laplacian is defined as
$$\nabla^{A*}\nabla^A:\Gamma(SM)\rightarrow\Gamma(SM),$$
where $\nabla^A:=\nabla-iA:\Gamma(SM)\rightarrow\Gamma(TM^*\otimes SM)$ and the action of $A$ on $\varphi$ is given by $A\otimes\varphi$.

The connection $\nabla^A$ is metric, so in particular
$$\frac{1}{2}\Delta|\varphi|^2=|\nabla^A\varphi|^2-\langle\nabla^{A*}\nabla^A\varphi,\varphi\rangle$$
holds.

Finally, the square of the magnetic Dirac operator satisfies (see for instance \cite{Nadine})
\begin{equation}\label{squaremagnetic}
\left( D_g-iA\cdot_g\right)^*\left( D_g-iA\cdot_g\right)=\nabla^{A*}\nabla^A+\frac{s_g}{4}-idA\cdot_g,
\end{equation}
where the action of $dA\cdot_g$ is locally defined as $$dA\cdot_g\varphi=\sum\limits_{j<k}dA(E_j,E_k)E_j\cdot_gE_k\cdot_g\varphi.$$

\subsection{Invariance properties}
In this section we recall some conformal transformations of the Dirac operator and the zero mode equation \eqref{Zero}.

Consider a conformal metric $g$, expressed as $g=h^{4/(n-2)}g_0$ for a smooth positive function $h$. There is a canonical isomorphism $F$ from  $S(M,g_0,\sigma)$ to $S(M,g,\sigma)$ which is a fiberwise isometry and preserves Clifford multiplication. In particular, for any one form $A$ the action $A\cdot_{g_0}\varphi$ is locally given by $A\cdot_{g_0}\varphi=\sum_{j}A(E_j)E_j\cdot_{g_0}\varphi$ for a local $g_0-$orthonormal frame. Then we have
$$F(A\cdot_{g_0}\varphi)=\sum_j A(E_j)F(E_j\cdot_{g_0}\varphi)=\sum_jA(E_j)(h^{\frac{-2}{n-2}}E_j)\cdot_{g}F(\varphi).$$
Thus $$F(A\cdot_{g_0}\varphi)=h^{\frac{2}{n-2}}A\cdot_gF(\varphi).$$

Under the isomorphism $F$, the Dirac operators corresponding to $g_0$ and $g$ are related by the formula
$$D_g(F(h^{-\frac{(n-1)}{n-2}}\varphi))=F(h^{-\frac{n+1}{n-2}}D_{g_0}\varphi) .$$

From this formula one can verify that the zero mode equation is conformally invariant. Assume $(\varphi,A)$ is a solution of equation with respect to the metric $g_0$.
Let $\psi:=h^{-(n-1)/(n-2)}F(\varphi)$. Then the pair $(\psi,A)$ is a solution of the equation with respect to metric $g$. Indeed,
\begin{align*}
D_g\psi&=F(h^{-\frac{n+1}{n-2}}D_{g_0}\varphi)\\
&=ih^{-\frac{n+1}{n-2}}F(A\cdot_{g_0}\varphi)=ih^{-\frac{n+1}{n-2}}\left(h^{\frac{2}{n-2}}A\cdot_gF(\varphi)\right)=iA\cdot_g\psi.
\end{align*}
Moreover, the $n/2$-norm for the 2-form $dA$ is the same for both metrics $g$ and $g_0$. Indeed, $$\int_M|dA|_g^{n/2}dv_g=\int_M\left(h^{\frac{-4}{n-2}}|dA|_{g_0}\right)^{n/2}dv_g=\int_M|dA|_{g_0}^{n/2}(h^{\frac{-2n}{n-2}}dv_g)=\int_M|dA|_{g_0}^{n/2}dv_{g_0}.$$

Another important invariance of the equation comes from gauge transformations. If $(\varphi,A)$ is a solution then $(e^{if}\varphi,A+df)$ is also a solution for any function $f$. So, without loss of generality we may assume that $div(A)=0$.

We summarize the previous discussion as follows:
\begin{lemma}
\begin{enumerate}
\item[(1)]The zero mode equation and the $n/2-$norm of the exterior derivative of the corresponding magnetic potential are conformally invariant.
\item[(2)]The zero mode equation is gauge invariant. In particular we may assume the magnetic potential  has vanishing  divergence.
\end{enumerate}
\end{lemma}

\section{Proof of the inequality}
In this section we prove the main inequality.  First, we establish some technical facts needed for the proof of Theorem 1.1.

We begin with the following integral identity

\begin{lemma}
For any nontrivial zero mode $(\varphi,A)$ we have
\begin{equation}\label{integral1}
0=\int_M|\varphi|_g^2\left\lbrace\frac{s_g}{4}-i\frac{\langle dA\cdot_g\varphi,\varphi\rangle_g}{|\varphi|_g^2}\right\rbrace dv_g+\int_M|\nabla^A \varphi|_gdv_g. 
\end{equation}

\end{lemma}
 
\begin{proof}
Since $\nabla^A=\nabla -iA$ is a metric connection we have $$0=\int_M-\Delta|\varphi|^2/2+\int_M|\nabla^A\varphi|^2u-\int_M\langle\nabla^{A*}\nabla^A\varphi,\varphi\rangle.$$

By the Lichnerowicz formula \eqref{squaremagnetic} $$(D-iA\cdot)^*(D-iA\cdot)=\nabla^{A*}\nabla^A+s_g/4-idA\cdot$$
and from the zero mode equation, it follows that $$-\nabla^{A*}\nabla^A\varphi=\frac{s_g}{4}\varphi-idA\cdot\varphi.$$
Hence the lemma follows.
\end{proof}

Now we consider the weighted linear  problem $$-a_n\Delta u+s_gu=4i\frac{\langle dA\cdot_g\varphi,\varphi\rangle_g}{|\varphi|_g^2}u.$$ 
Here $a_n=4(n-1)/(n-2)$.

From the previous integral identity, the real function $i\frac{\langle dA\cdot\varphi,\varphi\rangle}{|\varphi|^2}$  is positive on a set of positive measure. 
 
Therefore,the first positive eigenvalue of the weighted linear problem exists. We denote this eigenvalue as $\lambda_1\left(4i\frac{\langle dA\cdot_g\varphi,\varphi\rangle_g}{|\varphi|_g^2},g\right)$.

\begin{lemma}
The quantity $$\lambda_1\left(4i\frac{\langle dA\cdot_{g_0}\varphi,\varphi\rangle_{g_0}}{|\varphi|_{g_0}^2},g_0\right)$$ is conformally invariant.
\end{lemma}
\begin{proof}
Let $g=h^{4/(n-2)}g_0$ and let $F$ be the canonical isomorphism from $S(M,g_0,\sigma)$ to $S(m,g,\sigma)$ from Section 2.
We will prove that $\lambda_1$ with respect to the metric $g_0$ equals that with respect to the metric $g$.

Assume $(\varphi,A)$ is a nontrivial solution of the zero mode equation for the metric $g_0$. As in Section 2, the pair $(\psi,A)$ defined as $\psi=h^{-(n-1)/(n-2)}F(\varphi)$ is a solution of the equation with respect to the metric $g$.

We define the functional
$$I_{(g,\psi,A)}(u)=\dfrac{\int_M(a_n|\nabla u|_g^2+s_gu^2)dv_g}{4i\int_M \frac{\langle dA\cdot_g\psi,\psi\rangle_g}{|\psi|_g^2}  u^2dv_g  }.$$

First note that by the conformal transformation of the conformal Laplacian $$\int_M(a_n|\nabla u|_g^2+s_gu^2)dv_g=\int_M L_g(u)u\;dv_g=\int_Mh^{-(n+2)/(n-2)}L_{g_0}(hu)u\;dv_{g}.$$
Hence $$\int_M(a_n|\nabla u|_g^2+s_gu^2)dv_g=\int_ML_{g_0}(hu)hu\;dv_{g_0}$$
Second, $$|\psi|_{g}^2=h^{-2(n-1)/(n-2)}|\varphi|_{g_0}^2.$$

For a local $g_0-$orthonormal frame $\{E_j\}$  we consider the corresponding local $g-$orthonormal frame $\{\bar{E_j}:=h^{-2/(n-2)}E_j\}.$

Thus 
$$
dA\cdot_{g}\psi=\sum\limits_{j<k}dA(\bar{E_j},\bar{E_k})\bar{E_j}\cdot_{g} \bar{E_k}\cdot_{g}\psi.
$$
Since $$\langle \bar{E_j}\cdot_{g} \bar{E_k}\cdot_{g}F(\varphi),F(\varphi)\rangle_g=\langle E_j\cdot_{g_0} E_k\cdot_{g_0}\varphi,\varphi\rangle_{g_0}$$
we obtain $$\langle dA\cdot_{g}\psi,\psi\rangle_g=h^{-2(n-1)/(n-2)}\sum\limits_{j<k}dA(\bar{E_j},\bar{E_k})\langle E_j\cdot_{g_0} E_k\cdot_{g_0}\varphi,\varphi\rangle_{g_0}.$$
Then $$\langle dA\cdot_{g}\psi,\psi\rangle_g=h^{\frac{-4-2(n-1)}{n-2}}\langle dA\cdot_{g_0}\varphi,\varphi\rangle_{g_0}.$$
From this, we obtain 
$$\int_M \frac{\langle dA\cdot_g\psi,\psi\rangle_g}{|\psi|_g^2}  u^2dv_g =\int_M h^{-4/(n-2)}\frac{\langle dA\cdot_{g_0}\varphi,\varphi\rangle_{g_0}}{|\varphi|_{g_0}^2}  udv_{g}=\int_M \frac{\langle dA\cdot_{g_0}\varphi,\varphi\rangle_{g_0}}{|\varphi|_{g_0}^2}  (hu)^2dv_{g_0} $$
Therefore $$I_{(g,\psi,A)}(u)=I_{(g_0,\varphi,A)}(hu).$$

The lemma follows by taking the infimum in the previous equality. 
\end{proof}

\begin{claim}
$$\lambda_1\left(4i\frac{\langle dA\cdot_{g_0}\varphi,\varphi\rangle_{g_0}}{|\varphi|_{g_0}^2},g_0\right)\leq 1.$$
\end{claim} 
\begin{proof}
Let $u_1$ be the first eigenfunction associated to $\lambda_1\left(4i\frac{\langle dA\cdot_{g_0}\varphi,\varphi\rangle_{g_0}}{|\varphi|_{g_0}^2},g_0\right)$. 

To simplify notation, set $\lambda_1:=\lambda_1\left(4i\frac{\langle dA\cdot_{g_0}\varphi,\varphi\rangle_{g_0}}{|\varphi|_{g_0}^2},g_0\right)$.

Computing the scalar curvature of the metric $g=u_1^{4/(n-2)}g_0$ we have
$$s_{g}u_1^{4/(n-2)}=u_1^{-1}(-a_n\Delta u_1+s_{g_0}u_1)=\lambda_1\cdot \left(4i\frac{\langle dA\cdot_{g_0}\varphi,\varphi\rangle_{g_0}}{|\varphi|_{g_0}^2}\right).$$

So $$s_{g}=\lambda_1\cdot \left(u_1^{-4/(n-2)}4i\frac{\langle dA\cdot_{g_0}\varphi,\varphi\rangle_{g_0}}{|\varphi|_{g_0}^2}\right)=\lambda_1\cdot \left(4i\frac{\langle dA\cdot_{g}\psi,\psi\rangle_{g}}{|\psi|_{g}^2}\right),$$
where $\psi=u_1^{-(n-1)/(n-2)}F(\varphi)$.

Applying the integral identity \eqref{integral1} to the pair $(\psi,A)$ in the metric $g$ yields
\begin{equation}\label{integral2}
0=(\lambda_1-1)\int_M 4i\langle dA\cdot_g\psi,\psi\rangle_g +\int_M4|\nabla^{A}\psi|^2.
\end{equation}

If $\lambda_1-1>0$ then $\psi$ is trivial (and therefore $\varphi=0$).
\end{proof} 

The final fact needed for this preparatory part is the pointwise estimate: $$i\left\langle dA\cdot\varphi,\varphi\right\rangle\leq \left[\frac{n}{2}\right]^{1/2}|dA|\;|\varphi|^2.$$ 
Indeed, $$|i\left\langle dA\cdot\varphi,\varphi\right\rangle|\leq
\sum_{j<k}|dA(E_j,E_k)|\;\;|\langle E_j\cdot E_k\cdot\varphi,\varphi\rangle|. $$
On one hand, by Cauchy-Schwarz inequality we have $$|\langle E_j\cdot E_k\cdot\varphi,\varphi\rangle|\leq |\varphi|^2.$$
On the other hand, at a point $\left(dA(E_j,E_k)\right)$ is an antisymmetric matrix.  Hence we may choose the orthonormal frame $\{E_j\}$ such that this matrix is block diagonal with $\left[\frac{n}{2}\right]$ blocks of size $2\times 2$.
. The eigenvalues of each $2\times2$ block are purely imaginary. Hence $$\sum_{j<k}|dA(E_j,E_k)|\leq \left[\frac{n}{2}\right]^{1/2}\left(\sum_{j<k}|dA(E_j,E_k)|^2\right)^{1/2}=\left[\frac{n}{2}\right]^{1/2}|dA|.$$
 
Now we are in position to prove our inequality
\begin{proof}[Proof of Theorem 1.1]
The estimate $$i\left\langle dA\cdot\varphi,\varphi\right\rangle\leq \left[\frac{n}{2}\right]^{1/2}|dA|\;|\varphi|^2$$
together with Claim 1 implies that
$$1\geq \lambda_1\left(4i\frac{\langle dA\cdot\varphi,\varphi\rangle}{|\varphi|^2},g\right)\geq \inf_{u\in H^1(M)} \dfrac{\int_M a_n|\nabla u|^2+s_gu^2)dv_g}{\int_M 4\left[\frac{n}{2}\right]^{1/2}|dA|u^2dv_g}.$$
Let $v_n=\left[\frac{n}{2}\right].$

Applying the Hölder's inequality to the term $\int_M |dA|u^2dv_g$ we obtain $$\left(\int_M |dA|^{n/2}dv_g\right)^{2/n}\geq \dfrac{1}{4v_n^{1/2}}\inf_{u\in H^1(M)} \dfrac{\int_M a_n|\nabla u|^2+s_gu^2)dv_g}{\left(\int_M u^{2n/(n-2)}dv_g\right)^{(n-2)/n}}=\dfrac{1}{4v_n^{1/2}}Y(M^n,[g]). $$ 
Hence $$\left(\int_M |dA|^{n/2}dv_g\right)^{2/n}\geq\dfrac{1}{4v_n^{1/2}}Y(M^n,[g]).$$

Finally, we show that equality cannot hold.

By gauge invariance of the equation, we assume that $div(A)=0$.
Equality $$\left(\int_M |dA|^{n/2}dv_g\right)^{2/n}=\dfrac{1}{4v_n^{1/2}}Y(M^n,[g])$$
implies $$\lambda_1\left(4i\frac{\langle dA\cdot\varphi,\varphi\rangle}{|\varphi|^2},g\right)=1.$$
From integral identity \eqref{integral2} $$|\nabla^A\varphi|=0$$
and hence $\nabla\varphi=iA\otimes\varphi$.

Computing the magnetic Laplacian $\nabla^{A*}\nabla^A\varphi$ we get \begin{align*}
0&=\nabla^{A*}\nabla^A\varphi\\
&=\nabla^*\nabla \varphi+2i\nabla_A\varphi+idiv(A)\varphi+|A|^2\varphi\\
&=\nabla^*\nabla\varphi-|A|^2\varphi.
\end{align*}
Thus
$$|A|^2\varphi=\nabla^*\nabla\varphi=D^2\varphi-(s_g/4)\varphi.$$
and therefore $$\int_M|A|^2|\varphi|^2=\int_M|D\varphi|^2-\int_M\frac{s_g}{4}|\varphi|^2=\int_M|A|^2|\varphi|^2-\int_M\frac{s_g}{4}|\varphi|^2.$$
This forces $\varphi$ to be trivial.
\end{proof}

\end{document}